\documentclass[12pt]{article}
\usepackage{amsmath}
\usepackage{amscd}
\usepackage{amsthm}
\usepackage{amssymb} 
\usepackage{latexsym}
\usepackage{eufrak}
\usepackage{euscript}
\usepackage[width=17cm,height=22cm]{geometry}
\usepackage{epsfig}
\usepackage{tikz}
\usepackage{graphics}
\usepackage{array}
\usepackage{enumerate}
\usepackage{authblk}
\theoremstyle{theorem}
\newtheorem{theorem}{Theorem}[section]
\theoremstyle{corollary}
\newtheorem{corollary}{Corollary}[section]
\theoremstyle{lemma}

\theoremstyle{definition}

\theoremstyle{proof}
\theoremstyle{remark}
\newtheorem*{rem}{Remark}
\theoremstyle{example}
\newtheorem{example}{Example}[section]
 \def\e{\textbf{\textit{e}}}
  \def\s{\textbf{\textit{S}}} 
   \def\p{\textbf{\textit{P}}} 
    \def\b{\textbf{\textit{B}}}
     \def\a{\textbf{\textit{A}}} 
      \def\x{\textbf{\textit{x}}} 
       \def\y{\textbf{\textit{y}}}    
        \def\r{\mathcal{R}}     
       
\begin{document}
\title{An eigenvalue localization theorem for stochastic matrices and its application to  Randi\'c matrices}
\author[1,2]{\rm Anirban Banerjee}
\author[1]{\rm Ranjit Mehatari}
\affil[1]{Department of Mathematics and Statistics}
\affil[2]{Department of Biological Sciences}
\affil[ ]{Indian Institute of Science Education and Research Kolkata}
\affil[ ]{Mohanpur-741246, India}
\affil[ ]{\textit {\{anirban.banerjee, ranjit1224\}@iiserkol.ac.in}}
\maketitle
\begin{abstract}
A square matrix is called stochastic (or row-stochastic) if it is non-negative and has each row sum equal to unity. Here, we constitute an eigenvalue localization theorem for a stochastic matrix, by using its principal submatrices. As an application, we provide a suitable bound for the eigenvalues, other than unity, of the Randi\'c matrix of a connected graph.
\end{abstract}
\textbf{Keywards:} Stochastic matrix, eigenvalue localization, Randi\'c matrix, normalized Laplacian;
\section{Introduction}
Stochastic matrices occur in many fields of research, such as, computer-aided-geometric designs \cite{Pen}, computational biology \cite{New}, Markov chains \cite{Sene}, etc. A stochastic matrix $\s$ is irreducible if its underlying directed graph is strongly connected. 
In this paper, we consider $\s$ to be irreducible. 
 Let \textbf{\textit{e}} be the column vector whose all entries are equal to 1. Clearly, 1 is an eigenvalue of $\s$ with the corresponding  eigenvector \textbf{\textit{e}}. By Perron-Frobenius theorem (see Theorem 8.4.4 in \cite{Horn}), the multiplicity of the eigenvalue 1 is one and all other eigenvalues of $\s$ lie in the closed unit disc $\{z\in\mathbb{C}:|z|\leq1\}$. The eigenvalue 1 is called the Perron eigenvalue (or Perron root) of the matrix $\s$, whereas, the eigenvalues other than 1 are  known as  non-Perron eigenvalues of $\s$.\\
 
 Here, we describe a method for localization of the non-Perron eigenvalues of $\s$. The eigenvalue localization problem for stochastic matrices is not new. Many researchers gave significant contribution to this context \cite{Cve1,Kir1,Kir2,LiLi1,LiLi2}. In this paper, we use Ger{\v s}gorin disc theorem \cite{Ger} to localize the non-Perron eigenvalues of $\s$.  Cvetkovi\'c et al.~\cite{Cve1} and Li et al.~\cite{LiLi1,LiLi2} derived some usefull results, using the fact that any non-Perron eigenvalue of $\s$ is also an eigenvalue of the matrix $\s$-$(\e\e^T)diag(c_1,c_2,\cdots,c_n)$, where $ c_1,c_2,\cdots,c_n\in\mathbb{R} $. 
  
  In \cite{Cve1}, Cvetkovi\'c et al. found  a disc which contains all the non-Perron eigenvalues of $\s.$
  
  \begin{theorem}
\label{st:th3a}
\cite{Cve1}
Let $\s=[s_{ij}]$ be a stochastic matrix, and let
$s_i$ be the minimal element among the off-diagonal entries of the i-th column of $\s$. Taking $\gamma=\max_{i\in \mathbb{N}}(s_{ii}-s_i)$, for any $\lambda\in\sigma(\s)\setminus\{1\}$, we have
$$|\lambda-\gamma|\leq 1-trace(\s)+(n-1)\gamma.$$
\end{theorem} 

Theorem \ref{st:th3a} was further modified by Li and Li \cite{LiLi1}. They found another disc with different center and different radius.

\begin{theorem}
\label{st:th3b}
\cite{LiLi1}
Let $\s=[s_{ij}]$ be a stochastic matrix, and let
$S_i=\max_{j\neq i}s_{ji}$. Taking $\gamma'=\max_{i\in \mathbb{N}}(S_i-s_{ii})$, for any $\lambda\in\sigma(\s)\setminus\{1\}$, we have
$$|\lambda+\gamma'|\leq trace(\s)+(n-1)\gamma'-1.$$
\end{theorem}
In this paper, we show that there exist square matrices of order $n-1$, whose eigenvalues are the non-Perron eigenvalues of $\s$. We apply Ger{\v s}gorin disc theorem to those matrices in order to obtain our results. We provide an example where our result works better than  Theorem \ref{st:th3a} and Theorem \ref{st:th3b}. \\
 
 Let $\Gamma=(V,E)$ be a simple, connected, undirected  graph on $n$ vertices.  Two vertices $i,j\in V$ are called neighbours, written as $i\sim j$, if they are connected by an edge  in $E$. For a vertex
$i\in V$, let $d_i$
be its degree and $N_i$ be the set neighbours of the vertex $i$. For two vertices  $i,j\in V$, let $N(i,j)$ be the number of common neighbours of $i$ and $j$, that is, $N(i,j)=|N_{i}\cap N_{j}|$. Let $\textbf{\textit{A}}$ denote the adjacency matrix \cite{Cve} of $\Gamma$ and let $\textbf{\textit{D}}$ be the diagonal matrix of vertex degrees of $\Gamma$. The \textit{Randi\'c} matrix \textbf{\textit{R}} of $\Gamma$ is defined by $\textit{\textbf{R}}=\textbf{\textit{D}}^{-\frac{1}{2}}\textbf{\textit{A}}\textbf{\textit{D}}^{-\frac{1}{2}}$ which is  similar to the matrix $\mathcal{R}=\textbf{\textit{D}}^{-1}\textbf{\textit{A}}$. Thus, the matrices \textbf{\textit{R}} and $\mathcal{R}$ have the same eigenvalues. The matrix $\mathcal{R}$ is an irreducible stochastic matrix and its (i,j)-th entry is
 \begin{eqnarray*}
 \r_{ij}=\begin{cases}
             \frac{1}{{d_i}}, & \text{ if } i\sim j,\\
              0, & \text{ otherwise. }
             \end{cases}
\end{eqnarray*}             
The name \textit{Randi\'c matrix}  was introduced by Bozkurt et al.~\cite{Boz2} because \textbf{\textit{R}} has a connection with Randi\'c index \cite{Li,Ran}. In recent days, Randi\'c matrix becomes more popular to researchers. The Randi\'c matrix has a direct connection with \textit{normalized Laplacian matrix} $\mathcal{L}=\textbf{\textit{I}}_n-\textbf{\textit{R}}$ studied in \cite{Chung} and with $\Delta=\textbf{\textit{I}}_n-\r$ studied is \cite{Ban1,Ban2}. Thus, for any graph $\Gamma$, if $\lambda$ is an eigenvalue of the normalized Laplacian matrix, then $1-\lambda$ is an  eigenvalue of the Randi\'c matrix.\\

In Section 3, we localize non-Perron eigenvalues of $\mathcal{R}$. We provide an upper bound for the largest non-Perron eigenvalue and a lower bound for the smallest non-Perron eigenvalue of $\r$ in terms of common neighbours of two vertices and their degrees. The eigenvalue bound problem was studied previously in many articles \cite{Bau,Chung,LiGu,Rojo}, but the lower bound of the smallest eigenvalue of $\r$ given by Rojo and Soto \cite{Rojo} is the only one which involves the same parameters as in our bound. We recall the Rojo-Soto bound for Randi\'c matrix.
\begin{theorem}\cite{Rojo}
\label{st:th10}
Let $\Gamma$ be a simple undirected connected graph. If $\rho_n$ is the eigenvalue
with the largest modulus among the negative Randi\'c eigenvalues of $\Gamma$, then
\begin{equation}
\label{st:eq1}
|\rho_n|\leq1-\min_{i\sim j}\Big{\{}\frac{N(i,j)}{\max\{d_i,d_j\}}\Big{\}},
\end{equation}
where the minimum is taken over all pairs $(i, j)$ , $1\leq i<j\leq n$, such that the vertices i
and j are adjacent.
\end{theorem}
One of the drawbacks of Theorem \ref{st:th10} is that it always produces the trivial lower bound of $\rho_n$, if the graph contains an edge which does not participate in a triangle. Though the bound in Theorem \ref{st:th10} and our bound (Theorem \ref{st:th5}) are incomparable but, in many occasions,  our bound works better than Rojo-Soto bound. We illustrate this by a suitable example.

\section{Localization of the eigenvalues of an irreducible stochastic matrix}

Let $\e_1, \e_2,\ldots,\e_n$ be the standard orthonormal basis for  $\mathbb{R}^n$ and let $\e'=\left[\begin{array}{ccccc}
1&-1&-1&\cdots&-1
\end{array}\right]^T$. For $k\geq 1$, let $\textbf{\textit{j}}_k$ be the $k\times 1$ matrix with each entry equal to 1 and $\textbf{\textit{0}}_k$ be the $k\times 1$ zero matrix.  We define the matrix $\p$  as 
$$\p=\left[\begin{array}{ccccc}
\e&\e_2&\e_3&\ldots&\e_n
\end{array}\right].$$
It is easy to verify that the matrix $\p$ is nonsingular and its inverse  is $$\p^{-1}=\left[\begin{array}{ccccc}
\e'&\e_2&\e_3&\ldots&\e_n
\end{array}\right].$$
We use $\s(i|i)$ to denote the principal submatrix of $\s$ obtained by deleting $i$-th row and the $i$-th column. Now we have the following theorem.
\begin{theorem}
\label{st:th1}
Let $\s$ be a stochastic matrix of order n. Then $\s$ is similar to the matrix
 $$\left[\begin{array}{cc}
1&\x^T\\
\textbf{\textit{0}}_{n-1}&\b\end{array}\right]$$
where $\x^T=\left[\begin{array}{cccc}
s_{12}&s_{13}&\cdots&s_{1n}
\end{array}\right]$, and
$\b=\s(1|1)-\textbf{\textit{j}}_{n-1}\x^T.$
\end{theorem}
\begin{proof}
Let $\y=\left[\begin{array}{cccc}
s_{21}&s_{31}&\cdots&s_{n1}
\end{array}\right]^T$.
Then the matrices $\s$, $\p$, $\p^{-1}$ can be partitionoid as,
$$\s=\left[\begin{array}{cc}
s_{11}&\x^T\\
\y&\s(1|1)\end{array}\right],$$
$$\p=\left[\begin{array}{cc}
1&\textbf{\textit{0}}^T_{n-1}\\
\textbf{\textit{j}}_{n-1}&\textbf{\textit{I}}_{n-1} \end{array}\right],$$
$$\p^{-1}=\left[\begin{array}{cc}
1&\textbf{\textit{0}}^T_{n-1}\\
-\textbf{\textit{j}}_{n-1}&\textbf{\textit{I}}_{n-1} \end{array}\right].$$
Now 
\begin{eqnarray*}
\p^{-1}\s\p&=&\left[\begin{array}{cc}
1&\textbf{\textit{0}}^T_{n-1}\\
-\textbf{\textit{j}}_{n-1}&\textbf{\textit{I}}_{n-1}\end{array}\right]
\left[\begin{array}{cc}
s_{11}&\x^T\\
\y&\s(1|1)\end{array}\right]
\left[\begin{array}{cc}
1&\textbf{\textit{0}}^T_{n-1}\\
\textbf{\textit{j}}_{n-1}&\textbf{\textit{I}}_{n-1}\end{array}\right]\\
&=&\left[\begin{array}{cc}
s_{11}&\x^T\\
\y-s_{11}\textbf{\textit{j}}_{n-1}&\s(1|1)-\textbf{\textit{j}}_{n-1}\x^T\end{array}\right]
\left[\begin{array}{cc}
1&\textbf{\textit{0}}_{n-1}\\
\textbf{\textit{j}}_{n-1}&\textbf{\textit{I}}_{n-1}\end{array}\right]\\
&=&\left[\begin{array}{cc}
\sum_{j=1}^n s_{1j}&\x^T\\
\y-s_{11}\textbf{\textit{j}}_{n-1}+\s(1|1)\textbf{\textit{j}}_{n-1}-\textbf{\textit{j}}_{n-1}\x^T\textbf{j}_{n-1}&\s(1|1)-\textbf{\textit{j}}_{n-1}\x^T\end{array}\right].
\end{eqnarray*}
For $i=2,3,\ldots,n$, we have 
$(\p^{-1}\s\p)_{i1}=s_{i1}-s_{11}+\sum_{j=2}^ns_{ij} -\sum_{j=2}^ns_{1j}=0$ and hence the result follows. 
\end{proof}
\begin{theorem}
\label{st:th2}
Let $\s=[s_{ij}]$ be a stochastic matrix  of order n. Then any eigenvalue other than 1 is also an eigenvalue of the matrix 
$$\s(k)=\s(k|k)-\textbf{\textit{j}}_{n-1}\textbf{s}(k)^T,\textbf{ }k=1,2,\ldots, n$$
where $\textbf{s}(k)^T=\left[\begin{array}{cccccc}
s_{k1}&\cdots&s_{k,k-1}&s_{k,k+1}&\cdots&s_{kn}
\end{array}\right]$ is the k-deleted row of $\s$.
\end{theorem}
\begin{proof}
If $k=1$ then the proof is straightforward from Theorem  \ref{st:th1}.\\
For $k>1$, consider the permutation matrix $\p_k=\left[\begin{array}{cccccccc}
\e_2&\e_3&\cdots&\e_{k}&\e_1&\e_{k+1}&\cdots&\e_n
\end{array}\right]$. \\
Therefore, the matrix $\s$ is similar to the matrix 
\begin{eqnarray*}
\p_k^{-1}\s\p_k=\left[\begin{array}{cc}
s_{kk}&\x^T\\
\y&\s(k|k)\end{array}\right],
\end{eqnarray*}
where $\x=\textbf{s}(k)=\left[\begin{array}{cccccc}
s_{k1}&\cdots&s_{k,k-1}&s_{k,k+1}&\cdots&s_{kn}
\end{array}\right]^T$\\ and $\y=\left[\begin{array}{cccccc}
s_{1k}&\cdots&s_{k-1,k}&s_{k+1,k}&\cdots&s_{nk}
\end{array}\right]^T$.\\
Now, applying Theorem \ref{st:th1} to $\p_k^{-1}\s\p_k$, we get that $\s$ is similar to the matrix
 $$\left[\begin{array}{cc}
1&\textbf{s}(k)\\
\textbf{\textit{0}}_{n-1}&\s(k|k)-\textbf{\textit{j}}_{n-1}\textbf{s}(k)^T\end{array}\right].$$
Thus, any eigenvalue of $\s$, other than $1$, is also an eigenvalue of the matrix $\s(k)$, $k=1,2,\ldots,n$. 
\end{proof}
\begin{theorem}
\label{st:th3}
(\textbf{Ger{\v s}gorin}\cite{Ger})
Let $\a=[a_{ij}]$ be an $n\times n$ complex matrix. Then the eigenvalues of $\a$ lie in the region
$$G_\a=\bigcup_{i=1}^n\Big{\{}z\in\mathbb{C}:|z-a_{ii}|\leq \sum_{j\neq i}|a_{ij}|\Big{\}}.$$
\end{theorem}

\begin{theorem}
\label{st:th4}
Let $\s$ be a stochastic matrix of order n. Then the eigenvalues of $\s$ lie in the region 
$$\bigcap_{i=1}^n \Big{[}G_{\s(i)}\cup\{1\}\Big{]},$$
where 
$G_{\s(i)}=\bigcup_{k\neq i}\{z\in\mathbb{C}:|z-s_{kk}+s_{ik}|\leq\sum_{j\neq k}|s_{kj}-s_{ij}|\}.$
\end{theorem}

\begin{proof}
By Theorem \ref{st:th2}, we have, for all $i$,
$$\sigma(\s)=\sigma(\s(i))\cup\{1\}.$$
By Ger{\v s}gorin disc theorem, $\sigma(\s(i))\subseteq G_{\s(i)}$, for $i=1,2,\ldots,n$. 
Therefore, $$\sigma(\s)\subseteq\bigcap_{i=1}^n \Big{[}G_{\s(i)}\cup\{1\}\Big{]}.$$
Again, applying Theorem \ref{st:th3} to $G_{\s(i)}$, we get
\begin{eqnarray*}
G_{\s(i)}&=&\bigcup_{\substack{k=1,\\k\neq i}}^{n} \Big{\{}z\in\mathbb{C}:|z-\s(i)_{kk}|\leq \sum_{j\neq k}|\s(i)_{kj}|\Big{\}}\\
&=&\bigcup_{\substack{k=1,\\k\neq i}}^n\Big{\{}z\in\mathbb{C}:|z-s_{kk}+s_{ik}|\leq\sum_{j\neq k}|s_{kj}-s_{ij}|\Big{\}}.
\end{eqnarray*}
Hence, the proof is completed.
\end{proof}
\begin{rem}
Theorem \ref{st:th4} works nicely in some occasions even if  Ger{\v s}gorin disc theorem fails to provide a non-trivial result. For example, let $\s$  be an irreducible stochastic matrix with at least one diagonal element zero. Then, by Ger{\v s}gorin disc theorem, $G_\s\supseteq \{z\in\mathbb{C}:|z|\leq1\}$. But, in this case, Theorem \ref{st:th4} may provide a non-trivial eigenvalue inclusion set (see Example~\ref{st:ex1} and Example~\ref{st:ex2}). Again, Theorem~\ref{st:th3a} and Theorem~\ref{st:th3b} always provide  larger single discs, whereas, the eigenvalue inclusion set in Theorem~\ref{st:th4} is a union of smaller regions. Example~\ref{st:ex1} gives a numerical explanation to this interesting fact.
\end{rem}
\begin{example}
\label{st:ex1}Consider the $4\times 4$ stochastic matrix$$\s=\left[\begin{array}{cccc}
0.25&0.25&0.3&0.2\\
0&0.5&0.33&0.17\\
0.6&0.4&0&0\\
0.1&0.2&0.3&0.4
\end{array}\right].$$
Then we have $$\s(1)=\left[\begin{array}{ccc}
0.25&0.03&-0.03\\
0.15&-0.3&-0.2\\
-0.05&0&0.2
\end{array}\right],$$
$$\s(2)=\left[\begin{array}{ccc}
0.25&-0.03&0.03\\
0.6&-0.33&-0.17\\
0.1&-0.03&0.23
\end{array}\right],$$ 
$$\s(3)=\left[\begin{array}{ccc}
-0.35&-0.15&0.2\\
-0.6&0.1&0.17\\
-0.5&-0.2&0.4
\end{array}\right],$$ and
 $$\s(4)=\left[\begin{array}{ccc}
0.15&0.05&0\\
-0.1&0.3&0.03\\
0.5&0.2&-0.3
\end{array}\right].$$
\begin{figure}[h]
\label{st:fig1}
\centering
\includegraphics[width=\textwidth]{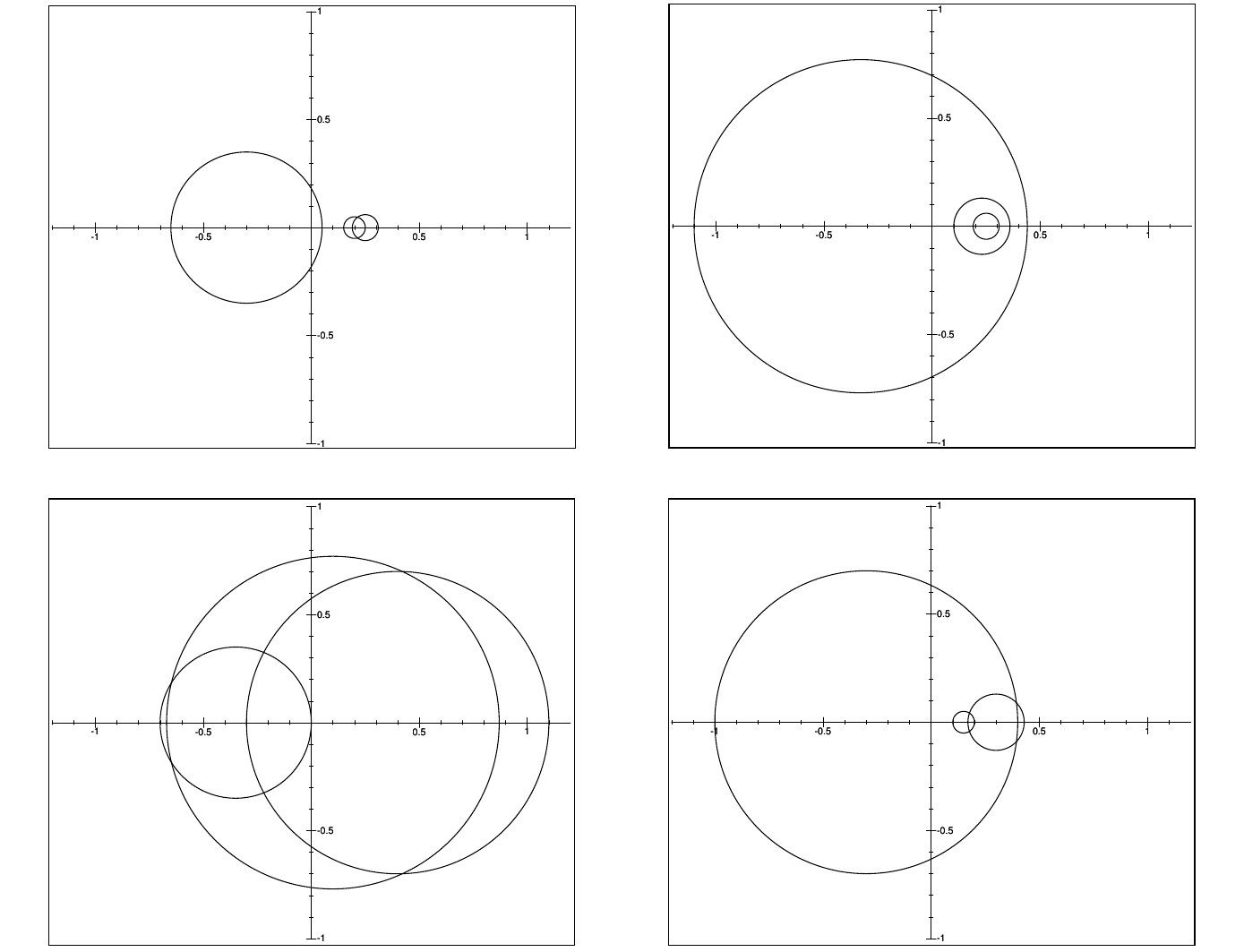}
\caption{The regions $G_\s(k)$, $k=1,2,3,4.$}
\end{figure}

The eigenvalues of $\s$ are $-0.307$, 0.174, 0.282, 1. Figure 1 shows that any eigenvalue other than 1 lies in each $G_{\s(k)}$. Also, from Figure 1, it is clear that $\sigma(\s)\subseteq \bigcap_{k=1}^4 [G_{\s(k)}\cup\{1\}]=G_{\s(1)}\cup \{1\}$. \\

Now, we estimate the eigenvalue inclusion sets in Theorem \ref{st:th3a} and Theorem \ref{st:th3b}. We have $s_1=0$, $s_2=0.2$, $s_3=0.3$, $s_4=0$ and $S_1=0.6$, $S_2=0.4$, $S_3=0.33$, $S_4=0.2$. Therefore,
$$\gamma=\max\{0.25,0.3,-0.3,0.4\}=0.4$$
and
$$\gamma'=max\{0.35,-0.1,0.33,-0.2\}=0.35.$$
By Theorem \ref{st:th3a}, any eigenvalue $\lambda\neq 1$ of $\s$ satisfies
$$|\lambda-0.4|\leq1.05.$$
Again, by Theorem \ref{st:th3b}, for any $\lambda\in\sigma(\s)\setminus\{1\}$, we have
$$|\lambda+0.35|\leq1.2.$$
It is easy to verify that $G_{\s(1)}$ is contained in both the discs.
Therefore, in this example, Theorem \ref{st:th4} works better than the other two.
\end{example}
\section{Bound for Randi\'c eigenvalues}
In this section, we give a nice bound for non-Perron eigenvalues of the Randi\'c matrix of  a connected graph $\Gamma$. Since \textbf{\textit{R}} is symmetric, the eigenvalues of \textbf{\textit{R}}(or $\r$) are all real and lie in the closed interval $[-1,1]$. We arrange the eigenvalues of $\r$ as $$-1\leq\lambda_n\leq\lambda_{n-1}\leq\cdots\leq \lambda_2<\lambda_1=1.$$ Now we have the following theorem.
\begin{theorem}
\label{st:th5}
Let $\Gamma$ be a simple connected graph of order n. Then 
$$-2+\max_{i\in\Gamma}\{\min_{k\neq i}\{\alpha_{ik}\},1\}\leq \lambda_n(\r)\leq\lambda_2(\r)\leq 2-\max_{i\in\Gamma}\{\min_{k\neq i}\{\beta_{ik}\},1\},$$
where, for $k\neq i$, $\alpha_{ik}$ and $\beta_{ik}$ are given by 
$$\alpha_{ik}=\begin{cases}
\frac{1}{d_k}+\frac{2N(i,k)}{\max\{d_i,d_k\}},&\textit{ if }k\sim i\\\
\frac{2N(i,k)}{\max\{d_i,d_k\}},&\textit{ if }k\nsim i
\end{cases}$$
and
$$\beta_{ik}=\begin{cases}
\frac{1}{d_k}+\frac{2}{d_i}+\frac{2N(i,k)}{\max\{d_i,d_k\}},&\textit{ if }k\sim i\\
\frac{2N(i,k)}{\max\{d_i,d_k\}},&\textit{ if }k\nsim i.
\end{cases}$$
\end{theorem}
\begin{proof}
Let $\lambda$ be a non-Perron eigenvalue of $\r$.
By Theorem \ref{st:th2}, $\lambda$  is also an eigenvalue of $\r(i)=\r(i|i)-\textbf{\textit{j}}_{n-1}\textbf{r}(i)^T$, where $\textbf{r}(i)^T$ is the $i$-deleted row of $\r$, for $i=1,2,\ldots,n$. So $\lambda$ lies in the regions $G_{\r(i)}$ with $$G_{\r(i)}=\bigcup_{k\neq i}\Big{\{}z\in\mathbb{C}:|z+r_{ik}|\leq\sum_{j\neq k}|r_{kj}-r_{ij}|\Big{\}}=\bigcup_{\substack{k=1\\k\neq i}}^nG_{\r(i)}(k),$$
where $G_{\r(i)}(k)$ are the Ger{\v s}gorin discs for $\r(i)$. Now, we consider each individual disc of $G_{\r(i)}$. For the vertex $k\in \Gamma$, $k\neq i$, we calculate the centre and the radius of $G_{\r(i)}(k)$. Here two cases may arise.\\
\textbf{Case I: }Let $k\sim i$. Then $r_{ik}=\frac{1}{d_i}$ and $r_{ki}=\frac{1}{d_k}$. Thus, the disc $G_{\r(i)}(k)$ is given by
\begin{eqnarray*}
|z+\frac{1}{d_i}|&\leq&\sum_{j\neq i,k}|r_{kj}-r_{ij}|\\
&=&\sum_{\substack{j\sim i,\\j\sim k}}|r_{kj}-r_{ij}|+\sum_{\substack{j\nsim i,\\j\sim k}}|r_{kj}-r_{ij}|+ \sum_{\substack{j\sim i,\\j\nsim k}}|r_{kj}-r_{ij}|+\sum_{\substack{j\nsim i,\\j\nsim k}}|r_{kj}-r_{ij}|\\
&=&N(i,k)|\frac{1}{d_k}-\frac{1}{d_i}|+\frac{d_k-N(i,k)-1}{d_k}+\frac{d_i-N(i,k)-1}{d_i}+0\\
&=&2-\frac{1}{d_k}-\frac{1}{d_i}-\frac{2N(i,k)}{\max\{d_i,d_k\}}.
\end{eqnarray*}
\textbf{Case II: }If $k\nsim i$. Then $r_{ik}=0$ and $r_{ki}=0$. Thus, we have the disc
\begin{eqnarray*}
|z|&\leq&\sum_{j\neq i,k}|r_{kj}-r_{ij}|\\
&=&\sum_{\substack{j\sim i,\\j\sim k}}|r_{kj}-r_{ij}|+\sum_{\substack{j\nsim i,\\j\sim k}}|r_{kj}-r_{ij}|+ \sum_{\substack{j\sim i,\\j\nsim k}}|r_{kj}-r_{ij}|+\sum_{\substack{j\nsim i,\\j\nsim k}}|r_{kj}-r_{ij}|\\
&=&N(i,k)|\frac{1}{d_k}-\frac{1}{d_i}|+\frac{d_k-N(i,k)}{d_k}+\frac{d_i-N(i,k)}{d_i}+0\\
&=&2-\frac{2N(i,k)}{\max\{d_i,d_k\}}.
\end{eqnarray*}
Now, we consider the whole region $G_{\r(i)}$. Since the eigenvalues of $\r$ are real, by combining  Case I and Case II, we obtain that any non-Perron eigenvalue $\lambda$ of $\r$ must satisfy
$$-2+\min_{k\neq i}\{\alpha_{ik}\}\leq \lambda\leq 2-\min_{k\neq i}\{\beta_{ik}\},$$
for all $i=1,2,\ldots,n.$\\



Therefore, by Theorem \ref{st:th4}, we obtain our required result. 
\end{proof}
\begin{corollary}
Let $\Gamma$ be a simple connected graph. If $\rho_2$ and $\rho_n$ are the smallest and the largest nonzero normalized Laplacian eigenvalue of $\Gamma$, then $$-1+\max_{i\in\Gamma}\{\min_{k\neq i}\{\beta_{ik}\},1\}\leq\rho_2\leq\rho_n\leq 3-\max_{i\in\Gamma}\{\min_{k\neq i}\{\alpha_{ik}\},1\}, $$
where $\alpha_{ik}$, $\beta_{ik}$ are the constants defined as in Theorem \ref{st:th5}.
\end{corollary}
\begin{corollary}
Let $\gamma$ be a connected $r$-regular graph on n vertices. If $\lambda\neq 1$ be any eigenvalue of $\r$, then 
$$-2+\frac{1}{r}\max_i\{\min_{k\neq i}\{\gamma_{ik}\},1\}\leq \lambda\leq2-\frac{1}{r}\max_i\{\min_{k\neq i}\{\delta_{ik}\},1\},$$
where 
\begin{equation*}
\gamma_{ik}=\begin{cases}
1+2N(i,k),&\textit{ if }k\sim i\\
2N(i,k),&\textit{ if }k\nsim i
\end{cases}
\end{equation*} 
and
\begin{equation*}
\delta_{ik}=\begin{cases}
3+2N(i,k),&\textit{ if }k\sim i\\
2N(i,k),&\textit{ if }k\nsim i.
\end{cases}
\end{equation*}
\end{corollary}

\begin{figure}[h]
\label{st:fig2}
\centering
\includegraphics[width=7cm]{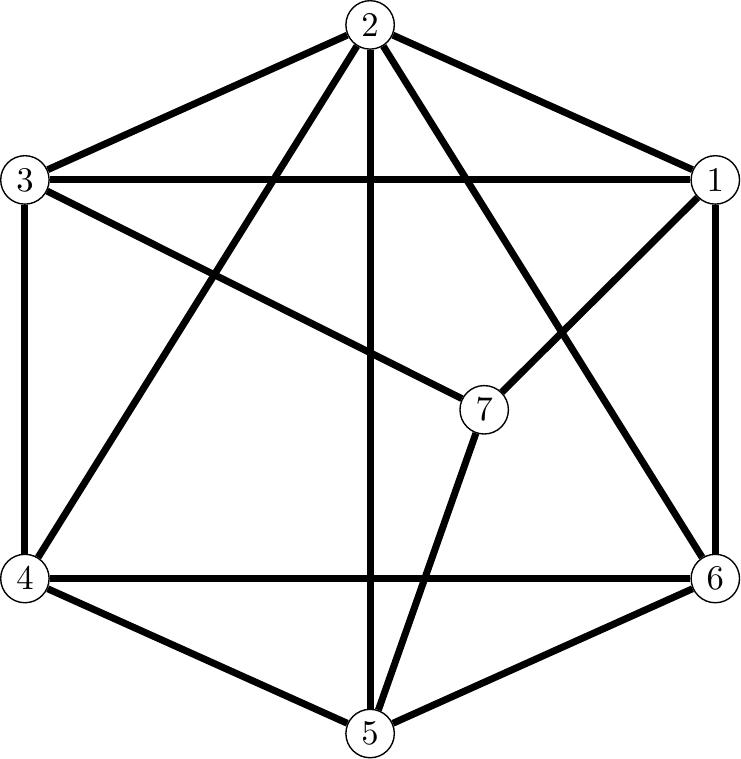}
\caption{A graph containing an edge which is not a part of a triangle. }
\end{figure}
Below we give an example where Theorem~\ref{st:th10} is improved by Theorem~\ref{st:th5}.
\begin{example}
\label{st:ex2}
Let $\Gamma$ be the graph as in Figure 2. The vertex degrees of $\Gamma$ are $d_1=4,$ $d_2=5,$ $d_3=d_4=d_5=d_6=4$, $d_7=3.$ The sets of neighbours of each vertex are given by
$$N_1=\{2,3,6,7\},$$
$$N_2=\{1,3,4,5,6\},$$
$$N_3=\{1,2,4,7\},$$
$$N_4=\{2,3,5,6\},$$
$$N_5=\{2,4,6,7\},$$
$$N_6=\{1,2,4,5\},$$
$$N_7=\{1,3,5\}.$$
Let $\alpha_i=\displaystyle \min_{k\neq i}\{\alpha_{ik}\}$ and $\beta_i=\displaystyle \min_{k\neq i}\{\beta_{ik}\}.$\\

The numbers of common neighbours of the vertex $2\in \Gamma$ with all other vertices are 
$N(2,1)=2$, $N(2,3)=2$, $N(2,4)=3$, $N(2,5)=2$, $N(2,6)=3$ and $N(2,7)=3$. Also note that the vertex $2$ is adjacent to all other vertices other than the vertex $7$. Thus we obtain 
\begin{eqnarray*}
\alpha_2&=&\min\Big{\{}\frac{1}{4}+\frac{4}{5},\frac{1}{4}+\frac{6}{5},\frac{6}{5}\Big{\}}\\
&=&1.05
\end{eqnarray*} and 
\begin{eqnarray*}
\beta_2&=&\min\Big{\{}\frac{1}{4}+\frac{2}{5}+\frac{4}{5},\frac{1}{4}+\frac{2}{5}+\frac{6}{5},\frac{6}{5}\Big{\}}\\
&=&1.2
\end{eqnarray*}
Similarly, for all other vertices of $\Gamma$ we get,
$\alpha_1=0.75$, $\beta_1=1.25$, $\alpha_3=0.75$, $\beta_3=1.25$, $\alpha_4=0.75$, $\beta_4=1$, $\alpha_5=0.333$, $\beta_5=0.833$, $\alpha_6=0.75$, $\beta_6=1$, $\alpha_7=1$, $\beta_7=1$.\\
 
Therefore, using Theorem~\ref{st:th5}, we get
$$\lambda_2\leq 0.75\textit{ and }\lambda_7\geq -0.95.$$
Note that, since $N(5,7)=0$, the lower bound for $\lambda_7$ in (\ref{st:eq1}) becomes $-1$.
\end{example}

\section{Acknowledgement}
We are very grateful to the referees for detailed comments and suggestions, which helped to improve the manuscript. We also thankful to Ashok K.~Nanda for his kind suggestions during writing the manuscript. Ranjit Mehatari is supported by CSIR, India, Grant No.~09/921(0080)/2013-EMR-I.

\end{document}